\documentclass[11pt, oneside]{amsart}   	
\usepackage{geometry}                		
\usepackage{graphicx}				
\usepackage{amssymb}

\newtheorem{thm}{Theorem}[section]

\newtheorem{fct}[thm]{Fact}

\newtheorem{cor}[thm]{Corollary}

\newtheorem{exmp}[thm]{Example}

\newtheorem{rmrk}[thm]{Remark}

\newtheorem{cons}[thm]{Construction}

\newtheorem{historic}[thm]{Historic Remark}

\newtheorem{conv}[thm]{Convention}

\newtheorem{prop}[thm]{Proposition}
\newtheorem{qus}{Question}
\newtheorem{lem}[thm]{Lemma}

\newtheorem{prob}{Problem}

\newtheorem{defn}{Definition}

\newcounter{contenumi}

\def\and{\mathrel{\&}}

\def\MM{{\revmathfont{M}}}

\newcommand{\revmathfont}[1]{{\textsf{#1}}}

\def\KP{\revmathfont{KP}}

\def\ZFC{\revmathfont{ZFC}}

\def\Inf{\revmathfont{Inf}}

\def\VeL{\revmathfont{V}=\revmathfont{L}}

\def\implies{\Rightarrow}

\title{Strong reducibilities and set theory}

\author{Noah Schweber}
\thanks{The author thanks Antonio Montalb\'an for many helpful conversations, as well as the anonymous referee for their helpful report. Part of this work appeared in chapter $8$ of the author's Ph.D. thesis \cite{Sch16}. The author was partially supported by NSF grant MSPRF-1606455.}				

\begin{document}
\maketitle

\begin{abstract}  We study Medvedev reducibility in the context of set theory --- specifically, forcing and large cardinal hypotheses. Answering a question of Hamkins and Li \cite{HaLi}, we show that the Medvedev degrees of countable ordinals are far from linearly ordered in multiple ways, our main result here being that there is a club of ordinals which is an antichain with respect to Medvedev reducibility. We then generalize these results to arbitrary ``reasonably-definable" reducibilities, under appropriate set-theoretic hypotheses.

We then turn from ordinals to general structures. We show that some of the results above yield characterizations of counterexamples to Vaught's conjecture; another applies to all situations, assigning an ordinal to any reasonable class of structures and ``measure" on that class. We end by discussing some directions for future research.
\end{abstract}

\tableofcontents

\section{Introduction}

If $A, B$ are two sets of natural numbers, we say $A$ is Turing reducible to $B$ if there is an algorithm for computing the characteristic function of $A$ with calls to the characteristic function of $B$; that is, if $\Phi_e^B=A$ for some Turing machine $\Phi_e$. Turing reducibility is generally believed to capture relative computability: if $A\le_T B$, then $B$ provides enough information to compute $A$, and (more importantly) conversely.

Extending this, for {\em structures } $\mathcal{A}$ and $\mathcal{B}$ we should say that $\mathcal{A}$ is reducible to $\mathcal{B}$ if we can ``effectively" build $\mathcal{A}$ from $\mathcal{B}$. This idea has multiple issues, however. First, there is no obvious sense in which a Turing machine can act directly on an abstract structure. We get around this by focusing on {\em copies}:

\begin{defn} A {\em copy} of a structure $\mathcal{A}$ is an isomorphic structure with domain $\omega$. 
\end{defn}

(All languages are assumed to be countable.)

We now intend to say that $\mathcal{A}$ is simpler than $\mathcal{B}$ if we can effectively build a {\em copy } of $\mathcal{A}$ from any {\em copy } of $\mathcal{B}$. This is much more satisfactory, as it is meaningful to feed a copy of a structure to a Turing machine as an oracle. However, this definition is still ambiguous: what kind of {\em uniformity } do we demand in the reduction? There are two opposite ``poles:"\footnote{Of course, these are not the only natural reducibility notions: {\em Medvedev-mod-parameters } is another natural reducibility notion, and turns out to lie {\em strictly between } strong and weak reducibility \cite{Kal12}. This refutes a natural (and apparently common) guess coming from the frequent utility of generic structures, since an easy forcing argument shows that whenever $\mathcal{A}\le_w\mathcal{B}$ there is some tuple $\overline{c}\in\mathcal{B}$ and $e\in\omega$ such that $e$ is a reduction from $B$ to a copy of $\mathcal{A}$ whenever $B$ is a {\em sufficiently generic } copy of $(\mathcal{B},\overline{c})$; Kalimullin's result shows that while the phrase ``sufficiently generic" can often be removed from arguments of this type, it is essential in this case. The question of when Muchnik and Medvedev-mod-parameters coincide in general is wide open.}
\begin{itemize}

\item {\bf (Muchnik (weak) reducibility) } $\mathcal{A}\le_w\mathcal{B}$ if for each copy $B$ of $\mathcal{B}$, there is some $e$ such that $\Phi_e^B$ is a copy of $\mathcal{A}$.

\item {\bf (Medvedev (strong) reducibility) } $\mathcal{A}\le_s\mathcal{B}$ if there is some $e\in\omega$ such that for each copy $B$ of $\mathcal{B}$, we have $\Phi_e^B\cong\mathcal{A}$.

\end{itemize} 

Muchnik reducibility has been studied extensively, and tends to be reasonably behaved. For instance, the foundational papers in $\alpha$-recursion theory led to a complete understanding of the Muchnik degrees of ordinals:

\begin{thm}[Sacks \cite{Sac76}, following \cite{Kri64}/\cite{Pla66}] Let $\alpha$ be an infinite ordinal. Then for an ordinal $\beta$, the following are equivalent:\footnote{Here, and throughout this paper, we conflate the ordinal $\eta$ with the linear order $(\eta, <)$.}\begin{itemize}

\item $\beta\le_w\alpha$.

\item $\beta$ is less than the least $\gamma>\alpha$ such that $L_\gamma\models$\KP+\Inf. {\em (This $\gamma$ is usually denoted by ``$\alpha^+$," but we will use ``$\omega_1^{CK}(\alpha)$" here to avoid clashing with the notation for the successor cardinal.)}

\end{itemize}
\end{thm}

By contrast, Medvedev reducibility of structures, while also a natural notion, has not received as much attention. It is clear, however, that even in the simple case of ordinals Medvedev reducibility is more complicated than Muchnik reducibility: for example, while if $\beta<\alpha$ we trivially have that any copy of $\alpha$ computes a copy of $\beta$ uniformly {\em modulo a parameter}, there is no immediate way to remove that parameter. Indeed in general this is impossible, and uniformity is a very strong requirement in this context.

As a trivial example of a Medvedev reduction, note that for $n\in\omega$ and $\alpha$ arbitrary we have $\alpha+n\ge_s\alpha$; this is because, given a copy of $\alpha+n$, we wait to see $n$-many elements enter above some element $x$, and when we see this we enumerate $x$ into the copy of $\alpha$ we're building. It takes some work to find nontrivial examples of strong reductions. To motivate their study, we briefly recall an example due to Hamkins and Li \cite{HaLi}; the proof we sketch is due to Montalb\'an.\footnote{A similar argument, replacing the single ``variable cut point" $b$ with a fixed finite number of such points, shows that $\omega_1^{CK}\cdot n+k\ge_s\omega_1^{CK}$ for all $n\in\omega\setminus\{0\}$ and all $k<\omega_1^{CK}$; meanwhile, the question of whether $\omega_1^{CK}\cdot \omega\ge_s\omega_1^{CK}$ is still open.}

\begin{exmp}\label{example} $\omega_1^{CK}\le_s\omega_1^{CK}+\omega$: Suppose $B$ is a copy of $\omega_1^{CK}+\omega$ and $b\in B$. Let $$T_b=\{\sigma\in (B^2)^{<\omega}: \sigma(\vert\sigma\vert-1)_0<\sigma(\vert\sigma\vert-2)_0<...<\sigma(0)_0<b<\sigma(\vert\sigma\vert-1)_1<\sigma(\vert\sigma\vert-2)_1<...<\sigma(0)_1\}$$ be the tree of parallel descending sequences in the cuts above and below $b$, and let $L_b$ be the Kleene-Brouwer ordering of $T_b$. Then $L_b$ is a computable ordinal, and $\sup\{L_b: b\in B\}=\omega_1^{CK}$, so if $B=\{b_0, b_1, ...\}$ then $\sum_{i\in\omega}L_{b_i}=\omega_1^{CK}$, and this is constructed uniformly in $B$.
\end{exmp}

The question of understanding the Medvedev reducibilities between ordinals was originally raised by Hamkins and Li \cite{HaLi}, who made a number of observations and asked a number of questions on the topic which roughly fell into two categories: ``fine" issues about the existence of reducibilities between specific ordinals or the effectiveness of specific operations, and ``coarse" issues about the overall degree structure $\mathcal{D}_s^{ord}$ of ordinals modulo Medvedev equivalence. Amongst the latter, Question 43 asked whether $\mathcal{D}_s^{ord}$ is linearly ordered. The first section of this article answers this question; the rest of this article generalizes the results of the first section and uses them to give a characterization of counterexamples to Vaught's conjecture.

\bigskip

Throughout this article, set-theoretic ideas and hypotheses play a fundamental role. For this reason, the main technical prerequisites for this article come from set theory, for which we recommend \cite{Jec03}. While not containing any technical prerequisites, to the interested reader we recommend \cite{AK00} for background on computable structure theory. Besides standard notation we write ``$Struc$" for the set of countable structures and ``$Mod(T)$" for the set of countable models of a theory $T$.

\section{The global structure of the Medvedev degrees of ordinals}

We begin by analyzing the global structure of the Medvedev degrees of ordinals, $\mathcal{D}_s^{ord}$. An important tool here will be the natural extension of Medvedev reducibility to generic extensions of the universe; its generalizations (Definition \ref{strongdef}) will be studied further in this article.

\begin{defn}\label{genmeddef} Given two structures $\mathcal{A}$ and $\mathcal{B}$ of arbitrary cardinality, we say $\mathcal{A}$ is {\em generically Medvedev reducible\footnote{Generic {\em Muchnik} reducibility has already been studied to a certain extent (see for example chapters 4-7 of the author's Ph.D. thesis \cite{Sch16}).}} to $\mathcal{B}$ (and we write ``$\mathcal{A}\le_s^*\mathcal{B}$") if $\mathcal{A}\le_s\mathcal{B}$ in every generic extension where each structure is countable. (By Shoenfield absoluteness, we may replace ``every" with ``some" in this definition, so generic Medvedev reducibility is genuinely a property of the structures rather than a set-theoretic coincidence.) We let $\mathcal{D}_{s^*}^{ord}$ be the (class-sized) degree structure of all ordinals, countable or not, under generic Medvedev reducibility.
\end{defn} 

With this in hand, we can give a strong affirmative answer to Question 43 of \cite{HaLi}:

\begin{thm}\label{finafriggingly} There is a club $A\subset\omega_1$ which is an $\le_s$-antichain.
\end{thm}

Before presenting the proof, we observe that we can immediately produce a two-element antichain under a mild large cardinal hypothesis --- namely, that $\Sigma^1_3$ sentences are reflected by forcing.\footnote{This is not such a wild hypothesis as it might appear at first: while the weakest natural large cardinal notion which outright {\em implies} it is rather strong, $\Sigma^1_3$-absoluteness itself has no additional consistency strength over \ZFC{}. To see this, let $\mathbb{P}_n$ force the $n$th ``forceable" $\Sigma^1_3$-sentence and consider the product $\prod\mathbb{P}_n$. By Shoenfield absoluteness this forces all forceable $\Sigma^1_3$ sentences, and so in particular forces $\Sigma^1_3$ generic absoluteness. See also the remark following the proof of Theorem $3$ of Bagaria and Friedman \cite{BF01}.} Besides helping to motivate the actual proof below, this type of argument will be used seriously later on: 

The $\le_s^*$-degrees are each countable and have the countable predecessor property (since there are only countably many indices for Medvedev reductions), hence we may find ordinals $\alpha,\beta<\omega_2$ which are $\le_s^*$-incomparable. Let $G$ be $Col(\omega, \alpha+\beta)$-generic over $V$; then $V[G]$ satisfies ``The $\le_s$-degrees are not linearly ordered," as witnessed by $\alpha$ and $\beta$. Appropriately expressed, this is $\Sigma^1_3$, hence already true in $V$ via our large cardinal assumption. 

We now want to improve this argument in two ways --- reduce the axioms needed and drastically increase the size of the antichain produced. The first of these is handled by the following absoluteness result:

\begin{lem}\label{medabs} Suppose $M$ is a countable transitive model of \ZFC, $\alpha,\beta\in ON^M$, and $e\in\omega$. Then the following are equivalent:\begin{enumerate}

\item $\alpha\le_s\beta$ via $\Phi_e$ (that is, $V\models$``$\alpha\le_s\beta$ via $\Phi_e$").

\item $M\models$``$\alpha\le_s^*\beta$ via $\Phi_e$".

\end{enumerate}

Similarly, $M\models$``$\alpha\le_s^*\beta$" iff $\alpha\le_s\beta$.
\end{lem}

\begin{proof} The final sentence follows from the main equivalence since $\omega^M=\omega^V$.

For $(1)\rightarrow(2)$, suppose that $e$ is a true index for a Medvedev reduction $\alpha\le_s\beta$, let $M[G]$ be a forcing extension in which $\alpha$ and $\beta$ are each countable, and let $A, B\in M[G]$ be copies $\alpha,\beta$ with domain $\omega$ respectively. Then in $V$ we have $\Phi_e^B\cong A$; since this is $\Sigma^1_1$ with real parameters from $M[G]$, we have $M[G]\models$``$\Phi_e^B\cong A$." So $M\models\alpha\le_s\beta$ via $\Phi_e$.

For $(2)\rightarrow(1)$, suppose $M\models$``$\alpha\le_s^*\beta$ via $\Phi_e$." Let $M[G]$ be a forcing extension in which $\alpha$ and $\beta$ are each countable, and let $A,B\in M[G]$ be copies of $\alpha,\beta$ with domain $\omega$ respectively. Let $\varphi\in\mathcal{L}_{\omega_1\omega}$ be the Scott sentence of $A$; this exists in $M[G]$. Then $M[G]$ satisfies ``For every isomorphic copy of $B$ with domain $\omega$, $\Phi_e^B$ is a structure satisfying $\varphi$." This is $\Pi^1_1$ with real parameters, so holds in $V$ --- that is, $\alpha\le_s\beta$ via $\Phi_e$.

Note that in each of the previous two paragraphs, since $M$ is countable an appropriately-generic $G$ does indeed exist (that is, $G\in V$).
\end{proof}

\begin{proof}[Proof of Theorem \ref{finafriggingly}] Let $\theta=\sup\{\gamma<\omega_1: \gamma\le_s^*\omega_1\}$, and note that $\theta$ is countable. For a set $X$, let $MC(X)$ denote the Mostowski collapse of $X$ and $mc_X: X\rightarrow MC(X)$ the associated collapse function. Let $$E=\{mc_M(\omega_1): \vert M\vert=\omega, M\prec H_{\omega_2}, \theta+1\subseteq M\}$$ be the set of possible collapses of $\omega_1$ via large enough countable elementary submodels of enough of the universe. $E$ contains a club (consider any $\omega_1$-indexed increasing sequence of appropriate models $(M_\eta)_{\eta<\omega_1}$ where limit models are given by unions and $\omega_1\cap M_\eta\subsetneq\omega_1\cap M_{\eta+1}$), so it is enough to show that $E$ is an antichain under $\le_s$.

Suppose $\alpha<\beta$ are in $E$, and let $M_\beta\prec H_{\omega_2}$ with $mc_{M_\beta}(\omega_1)=\beta$. Then $M_\beta\cap \omega_1=\beta$ and $mc_{M_\beta}$ is the identity on $M_\beta\cap\omega_1$, since the countable ordinals in $M_\beta$ are closed downwards.  Since admissibility is absolute, the next admissible $\omega_1^{CK}(\alpha)$ is in $M_\beta$ and so is $<\beta$; but then since Medvedev reducibility refines Muchnik reducibility we have $\sup\{\gamma: \gamma\le_s\alpha\}\le\omega_1^{CK}(\alpha)<\beta$, so $\alpha\not\ge_s\beta$.

In the other direction, $M_\beta\models \omega_1\not\ge_s^*\alpha$ since $\alpha\in M_\beta$ and $\alpha>\theta$ (since every element of $E$ is $>\theta$). This gives us $MC(M_\beta)\models mc_{M_\beta}(\omega_1)\not\ge_s^* mc_{M_\beta}(\alpha)$. By definition $mc_{M_\beta}(\omega_1)=\beta$, and since $\alpha<\beta$ we have $mc_{M_\beta}(\alpha)=\alpha$; so $MC(M_\beta)\models \beta\not\ge_s^*\alpha$. Now apply Lemma \ref{medabs}.
\end{proof}

\begin{rmrk} Another approach to building $\le_s$-antichains is as follows. Since the $\le_s^*$-degrees have the countable predecessor property, we can find $\alpha,\beta\in Ord$ with $\alpha\not\le_s^*\beta,\beta\not\le_s^*\alpha$; in the Mostowski collapse of a countable elementary substructure of (enough of) $V$ containing $\alpha$ and $\beta$, the images of $\alpha$ and $\beta$ are $\le_s^*$-incomparable, and by Lemma \ref{medabs} are in fact $\le_s$-incomparable. 

This also allows us to build $\le_s$-antichains whose ordertype in the usual sense is arbitrarily large below $\omega_1$: via Erdos-Rado, there is an $\le_s^*$-antichain $A$ of cardinality $\omega_1$; now given $\alpha<\omega_1$, letting $M$ be a countable elementary submodel with $\alpha+1\subseteq M$ and $A\in M$ and taking $N$ to be the Mostowski collapse of $M$, the image of $A\cap M$ under the collapse map is an $\le_s$-antichain of ordertype (in the usual sense) $>\alpha$.

There is no reason why these countable antichains should cohere nicely, however, so this does not let us prove Theorem \ref{finafriggingly} or even the existence of a merely uncountable $\le_s$-antichain; however, it does seem potentially more flexible for other situations, so we mention it here.
\end{rmrk}

Incidentally, by elementarity together with Lemma \ref{medabs}, we get a couple further homogeneity properties of $E$ worth noting.

\begin{cor} There is a countable ordinal $\mu_{club}$ such that for club-many countable $\alpha$, $\min\{\gamma: \gamma\not\le_s\alpha\}=\mu_{club}$; and $\mu_{club}=\min\{\gamma: \gamma\not\le_s^*\omega_1\}$
\end{cor}

\begin{proof} Note that $\min\{\gamma: \gamma\not\le_s^*\omega_1\}\le\theta$, and every $\alpha\in E$ is the Mostowski collapse of $\omega_1$ via some countable elementary submodel containing $\theta+1$ as a subset; so $\min\{\gamma: \gamma\not\le_s^*\omega_1\}$ is fixed by the Mostowski collapse of such a model. Now apply Lemma \ref{medabs}.
\end{proof}

The second corollary requires a couple definitions:

\begin{defn} Write $\MM(\mathcal{A})$ for the set of {\em indices for Medvedev computations of ordinals} from $\mathcal{A}$, that is, the set of $e$ such that for every pair of copies $A_0, A_1$ of $\mathcal{A}$, $\Phi_e^{A_0}$ and $\Phi_e^{A_1}$ are total and yield isomorphic ordinals (or such that the previous holds in every generic extension where $\mathcal{A}$ is countable).
\end{defn} 

When $c\in\MM(\mathcal{A})$ we may write ``$\Phi_c^\mathcal{A}$" without confusion.

\begin{defn} The {\em collapsed spectrum} of an ordinal $\alpha$, denoted by ``$CS(\alpha)$," is the ordertype of the set $\{\gamma:\gamma\le_s^*\alpha\}$.
\end{defn}

\begin{cor}\label{thetaclub} There is a countable ordinal $\theta_{club}$ such that $CS(\alpha)=\theta_{club}$ for club-many countable $\alpha$; and $\theta_{club}=CS(\omega_1)$.
\end{cor}

\begin{proof} For $\alpha$ a countable ordinal, consider the preordering $\trianglelefteq_\alpha$ on $\omega$ given by comparing lengths of Medvedev reductions from $\alpha$:\begin{itemize}

\item If $c\not\in \MM(\alpha)$, then $c\trianglelefteq_\alpha d$.

\item If $c, d\in \MM(\alpha)$, then $c\trianglelefteq_\alpha d$ iff ``$\Phi_c^{\alpha}\le\Phi_d^{\alpha}$" is true in some/every generic extension in which $\alpha$ is countable.
\end{itemize}

Identically to the proof of Lemma \ref{medabs}, both $\MM(\alpha)$ and $\trianglelefteq_\alpha$ are absolute between generic extensions. The ordertype of $\trianglelefteq_\alpha$ is exactly $CS(\alpha)$, so for every $\alpha\in E$ we have $CS(\alpha)=CS(\omega_1)$.
\end{proof}

\begin{rmrk} Note that Corollary \ref{thetaclub} alone would give us a club antichain: letting $C$ be a club on which $CS(-)$ is constant, consider the sub-club $C'=\{\gamma\in C: \forall \beta<\gamma, \gamma>\omega_1^{CK}(\beta)\}$. Given $\alpha<\beta$ in $C'$, we cannot have $\alpha\ge_s\beta$ since $\beta>\omega_1^{CK}(\alpha)$, and if $\beta\ge_s\alpha$ we would have to have $CS(\beta)>CS(\alpha)$ since $\beta>\sup\{\gamma:\gamma\le_s\alpha\}$. We will use this below, in Theorem \ref{clubordasr}, to construct a club antichain with respect to more complicated reducibilities using additional hypotheses.
\end{rmrk}

Finally, we can bound the ``height" of $\mathcal{D}_{s^*}^{ord}$:

\begin{fct}\label{thetasupmed} For each $\alpha\in ON$, we have $CS(\alpha)<\omega_1^L$.
\end{fct}

\begin{proof} Similarly to the proof of Lemma \ref{medabs}, $L$ interprets ``$e$ is an index witnessing $\alpha\ge_s^*\beta$" correctly, even if $\alpha, \beta$ are uncountable. Hence $CS(\alpha)=CS(\alpha)^L$; the latter is a countable ordinal in $L$, hence $<\omega_1^L$.
\end{proof}

Under a mild set-theoretic hypothesis, the absoluteness of $\le_s^*$ gives the following bounds:

\begin{cor}\label{NoMedOm1} Assume $\omega_1^L<\omega_1$. Then:\begin{itemize}

\item There is a countable ordinal $\alpha$ such that $CS(\gamma)<\alpha$ for all $\gamma\in ON$.

\item There is a countable ordinal $\alpha$ such that for every set $A\subset ON$ which is an $\le_s^*$-chain, the transitive collapse of $A$ is $<\alpha$.

\item There is no embedding of $\omega_1$ into $\mathcal{D}_{s^*}^{ord}$ (and hence $\mathcal{D}_{s^*}^{ord}$ is not countably directed).

\end{itemize} 
\end{cor}

\begin{proof} The first follows directly from the fact above by taking $\alpha=\omega_1^L$. 

For the second, suppose $A$ were an $\le_s^*$-chain of ordertype $\omega_1^L+1$ in the usual ordering of $ON$. By thinning $A$ if necessary, we may assume that $\alpha_0,\alpha_1\in A,\alpha_0<\alpha_1\implies \omega_1^{CK}(\alpha_0)<\alpha_1$. But then $CS(\max(A))=\omega_1^L$, contradicting the fact above.

For the third, if $f:\omega_1\rightarrow\mathcal{D}_{s^*}^{ord}$ is increasing, let $g:\omega_1\rightarrow\omega_1$ be given by $g(\gamma)=\min(f(\gamma))$; without loss of generality, we may assume $g$ is increasing. But then $CS(g(\omega_1^L+1))\ge\omega_1^L$, a contradiction.
\end{proof}







\bigskip

The results above leave open, of course, the precise values of $\theta_{club}$ and $\theta_{sup}:=\sup\{CS(\alpha): \alpha\in ON\}$, even under large cardinals. The most conspicuous dangling thread, however, is the role of set theory in the proof of Corollary \ref{NoMedOm1} above. While the existence of large antichains is provable in \ZFC, it seems possible that the nonexistence of long chains may require additional hypotheses. As evidence for this, we observe that for reducibilities other than Medvedev --- even simple and natural ones --- \ZFC{} alone is indeed insufficient for the analogous result: 

Specifically, consider {\em hyper-Medvedev } reducibility, namely being uniformly $\Delta^1_1$-definable in every copy; write ``$\le_{sH}$" and ``$\mathcal{D}_{sH}^{ord}$" for the reducibility and the corresponding degree structure, respectively. This is a natural reducibility to consider in this context given the relationship between hyperarithmetic and computable in the context of ordinals; also, hyper-Medvedev reducibility is $\Pi^1_2$, so no more complex than Medvedev reducibility from a descriptive set theoretic standpoint. Finally, hyper-Medvedev reducibility is {\em structurally } well-behaved: it is easy to see that finite joins exist in $\mathcal{D}_{sH}^{ord}$, as follows. Example \ref{example} generalizes to show that $\lambda_0+\lambda_1\ge_{s}\lambda_1$ whenever $\lambda_0>\lambda_1$ are indecomposable. Now, if $\lambda_0>\lambda_1$ are indecomposable, then $a$ is in the $\lambda_0$-part of $\lambda_0+\lambda_1$ iff $[0, a]+\lambda_1<\lambda_0+\lambda_1$. Since we can compare ordinals in a uniformly $\Delta^1_1$ way, this means that $\lambda_0+\lambda_1\ge_{sH}\lambda_0$ whenever $\lambda_0>\lambda_1$ are indecomposable. Put together, this gives us that the $\le_{sH}$-join of ordinals is given by their natural (or Hessenberg) sum. However, this argument fails when we try to work with Medvedev reducibility instead, since we no longer obviously have $\lambda_0+\lambda_1\ge_s\lambda_1$, and it is not currently known whether finite joins exist in $\mathcal{D}_s^{ord}$.

So hyper-Medvedev is an attractive alternative to Medvedev reducibility, at least in the ordinals context. The connection with the results above is this. Theorem 8 of Hamkins, Linetsky, and Reitz \cite{HLR13} states that the set of ordinals $\alpha$ such that every element of $L_\alpha$ is definable in $L_\alpha$ without parameters is unbounded below $\omega_1^L$. It is not hard to show that if $\beta<\alpha$ is parameter-freely definable in $L_\alpha$ then $\beta\le_{sH}\alpha$. Then assuming $\omega_1^L=\omega_1$, the set of  $\alpha$ such that every element of $L_\alpha$ is parameter-freely definable in $L_\alpha$ forms a chain of ordertype $\omega_1$ in $\mathcal{D}_{sH}^{ord}$. Since hyper-Medvedev reducibility is $\Pi^1_2$, this means that there can be no coarse definability-theoretic proof from \ZFC{} alone that $\mathcal{D}_s^{ord}$ does not embed $\omega_1$.\footnote{The proof of Theorem 42 of \cite{HaLi}, which would contradict this observation, is flawed: it relies on a certain descending sequence of sets having nonempty intersection, which need not happen. The authors will correct this in a future draft.} So it is plausible that $\omega_1$ could embed into $\mathcal{D}_s^{ord}$ consistently in \ZFC.




\section{Absolute strong reducibilities} Several of the results of the previous section invoked set theoretic principles beyond \ZFC, either to bring a result from the ``generic" context down to apply to genuinely countable ordinals or to bring out combinatorial features which otherwise wouldn't exist (namely, the appropriate club dichotomy), and that section ended with an example of a natural strong reducibility, of the same complexity as Medvedev reducibility, which can consistently fail to have some of the global structural properties we hope that Medvedev reducibility might satisfy in \ZFC{} alone. This tells us that \ZFC{} alone will generally not suffice to give satisfying global descriptions of the degrees of countable ordinals under a natural strong reducibility notion.

On the other hand, our use of additional set-theoretic hypotheses was ``coarse": in general, arbitrary strong reducibilities can be tamed by appropriate large cardinal axioms. In this section we treat specifically the projectively-definable case, but it is clear how to extend to broader contexts under stronger hypotheses. The main value of these arguments is in the following section, where they will be used to provide characterizations of counterexamples to Vaught's conjecture.

\bigskip

By a {\em strong reducibility notion} we mean a subset $\mathfrak{R}$ of $\omega\times Struc\times Struc$; we interpret ``$\mathfrak{R}(e, \mathcal{A},\mathcal{B})$" as ``$e$ is an index for a reduction from $\mathcal{A}$ to $\mathcal{B}$," and say ``$\mathcal{A}\le_\mathfrak{R}\mathcal{B}$" if $\exists e(\mathfrak{R}(e, \mathcal{A},\mathcal{B}))$; sometimes we refer to $\le_\mathfrak{R}$ itself as the strong reducibility, when no confusion will arise (note however that strictly speaking $\le_\mathfrak{R}$ does not determine $\mathfrak{R}$).

Suppose $\mathfrak{R}\subseteq\omega\times Struc\times Struc$ is definable --- by which we shall mean, definable in the language of set theory with real parameters. Then there are two ``generic" versions of $\mathfrak{R}$, applicable to all structures (not just the countable ones):\begin{itemize}

\item $\mathfrak{R}^{*, +}(e, \mathcal{A},\mathcal{B})$ if for some generic extension $V[G]$, we have $V[G]\models\mathfrak{R}(e, \mathcal{A},\mathcal{B})$ (and in particular $\mathcal{A},\mathcal{B}$ are countable in $V[G]$). We denote the induced preorder by ``$\le_\mathfrak{R}^{*, +}$."

\item $\mathfrak{R}^{*, -}(e, \mathcal{A}, \mathcal{B})$ if for every generic extension $V[G]$ in which $\mathcal{A},\mathcal{B}$ are countable, we have $V[G]\models\mathfrak{R}(e, \mathcal{A}, \mathcal{B})$. We denote the induced preorder by ``$\le_\mathfrak{R}^{*, -}$."

\end{itemize}

(Note that these are really generic versions of {\em some fixed definition } of $\mathfrak{R}$; when we speak of a definable reducibility, we really mean a fixed definition of that reducibility.)

When $\mathfrak{R}$ is Medvedev reducibility, these coincide by Shoenfield absoluteness; but in general they need not (e.g. take $\mathfrak{R}$ to depend on the continuum hypothesis). However, an ameliorating theme in in descriptive set theory is that under large cardinal assumptions, ``reasonably definable" objects behave absolutely. At the most basic level, if $\mathfrak{R}$ is reasonably definable and appropriate large cardinal assumptions hold, we expect $\le_\mathfrak{R}^{*, +}$ and $\le_\mathfrak{R}^{*, -}$ to coincide; but more generally we expect reasonable sentences about such an $\mathfrak{R}$, such as basic properties of the corresponding degree structures, to be forcing-invariant.

So we will restrict our attention to ``absolute strong reducibilities." Here there is a slight issue --- it quickly becomes very tedious to state all theorems optimally. To facilitate things, we will adopt an excessively-broad definition of ``absolute strong reducibility," which will in turn require us to assume an excessively-powerful large cardinal axiom; we leave the straightforward improvements to the reader interested in minimizing set-theoretic assumptions. Following Montalb\'an \cite{Mon13}, we will look at {\em projective } definability.

\begin{defn}\label{strongdef} An {\em absolute strong reducibility} (ASR) $\mathfrak{R}$ is a subset of $\omega\times Struc\times Struc$ such that the following hold\begin{itemize}

\item $\mathfrak{R}$ is projectively definable (or rather, the predicate ``$a, b$ are codes for structures $\mathcal{A},\mathcal{B}$ and $\mathfrak{R}(e, \mathcal{A},\mathcal{B})$ holds" is projective in $e, a, b$).

\item {\bf (Injectivity) } For each $e$ and $\mathcal{B}$, there is at most one $\mathcal{A}$ with $\mathfrak{R}(e,\mathcal{A},\mathcal{B})$.

\item {\bf (Transitivity) } If $\mathcal{A}\le_\mathfrak{R}\mathcal{B}\le_\mathfrak{R}\mathcal{C}$, then $\mathcal{A}\le_\mathfrak{R}\mathcal{C}$.
\end{itemize}
(We could include more axioms, like reflexivity, but will have no need for it.).
\end{defn}

ASRs are generally well-behaved (and in particular deserve their name) assuming that large cardinals exist.

\begin{conv} For the rest of this article we assume the large cardinal hypothesis ``There is a proper class of Woodin cardinals." The relevant consequences of this hypothesis are the following:\begin{itemize}

\item The truth values of projective sentences with real parameters are not changed by forcing. In particular, $\mathfrak{R}^{*, +}=\mathfrak{R}^{*, -}$ for every ASR $\mathfrak{R}$.

\item The projective club dichotomy: every projectively definable set of countable ordinals contains or is disjoint from a club. (This is used exclusively in Theorem \ref{clubordasr}.)
\end{itemize}

We write ``ASR" instead of ``PSR" to emphasize that it is the absoluteness of $\mathfrak{R}$, not its projectiveness per se, that really matters. 
\end{conv}

\begin{rmrk} A proper class of Woodin cardinals is significant overkill for our purposes. This is visible already at the level of consistency strength, since the consistency strength of projective absoluteness for forcing extensions is equiconsistent with infinitely many strong cardinals (see Hauser \cite{Hau95}), and the projective club dichotomy is even weaker --- it is consistent relative to a single ineffable cardinal, which is a large cardinal property compatible with \VeL{} (see section $5$ of Harrington \cite{Ha78}). Additionally, the assumption of a proper class of Woodin cardinals in fact would allow us to consider not just {\em projective} strong reducibilities but strong reducibilities defined inside the inner model $L(\mathbb{R})$. This is a significantly broader class of reducibility notions, but is also more technical to define.
\end{rmrk}

For an arbitrary ASR $\mathfrak{R}$, we build a picture similar to that of Medvedev reducibility. We define $\mathcal{D}_\mathfrak{R}^{ord}$, $\mathcal{D}_{\mathfrak{R}*}^{ord}$, $\MM_\mathfrak{R}(\mathcal{A})$, and $CS_\mathfrak{R}(\mathcal{A})$ analagously to their Medvedev versions in Section 2, and show that the results of the previous sections hold for arbitrary ASRs.

A couple results generalize easily: 

\begin{thm} Let $\mathfrak{R}$ be an ASR. Then:\begin{enumerate}

\item There is a (least) $\theta_{sup}^\mathfrak{R}<\omega_1$ such that $CS_\mathfrak{R}(\alpha)<\theta_{sup}^\mathfrak{R}$ for all ordinals $\alpha$.

\item There is no embedding of $\omega_1$ into $\mathcal{D}_\mathfrak{R}^{ord}$.
\end{enumerate}
\end{thm}

\begin{proof} $(1)$ is proved via absoluteness as in Corollary \ref{NoMedOm1}. For each $\alpha\in Ord$, we have $CS_\mathfrak{R}(\alpha)<\omega_1$; since $\mathfrak{R}$ is absolute, after collapsing $\omega_1$ we have $CS_\mathfrak{R}(\alpha)<\omega_1^V$, hence is countable, so by absoluteness again we have $CS_\mathfrak{R}(\alpha)<\omega_1$ in the ground model. 

$(2)$ follows from $(1)$ as before: given an embedding of $\omega_1$ into $\mathcal{D}_\mathfrak{R}^{ord}$, we can find an $\alpha$ with $CS_\mathfrak{R}(\alpha)\ge\theta_{sup}$.
\end{proof}

By contrast, when we try to extend Theorem \ref{finafriggingly} we run into a slight problem --- the argument used absoluteness of the relevant properties between arbitrary transitive models (Lemma \ref{medabs}), while our large cardinal assumption only gives us much more restricted forms of absoluteness (e.g. between forcing extensions). Instead, we use the club dichotomy provided by our large cardinal assumption to find a sufficiently homogeneous set of countable ordinals, and then argue that it constitutes an antichain:

\begin{thm}\label{clubordasr} Let $\mathfrak{R}$ be an ASR. Then there is a club $A\subset\omega_1$ which is an $\mathfrak{R}$-antichain.
\end{thm}

\begin{proof} Fix an ASR $\mathfrak{R}$. For $\alpha<\omega_1$, we define the binary relation $O_\alpha\subseteq \omega\times\omega$ as follows. $O_\alpha(c, d)$ holds iff one of the two possibilities holds:\begin{enumerate}

\item $c\not\in\MM_\mathfrak{R}(\alpha)$;

\item $c, d\in\MM_\mathfrak{R}(\alpha)$ and $\gamma_0<\gamma_1$ where $\mathfrak{R}(c,\gamma_0, \alpha)$ and $\mathfrak{R}(d, \gamma_1,\alpha)$.
\end{enumerate}

$O_\alpha$ is a preorder, and it is easy to check that its ordertype is exactly that of $CS_\mathfrak{R}(\alpha)$. Note that if $\beta\le_\mathfrak{R}\alpha$ then $CS_\mathfrak{R}(\beta)\le CS_\mathfrak{R}(\alpha)$. Both $(1)$ and $(2)$ are projective, hence by our large cardinal assumption for each $m,n\in\omega$ the set $\{\alpha\in\omega_1: O_\alpha(m,n)\}$ either contains or is disjoint from a club $C_{m, n}$. Let $D=\bigcap_{m, n\in\omega}C_{m, n}$. $D$ is a club, and for each $\gamma_0, \gamma_1\in D$ we have $CS(\gamma_0)=CS(\gamma_1)$. 

Let $S(\beta)=\sup\{\gamma: \gamma\le_\mathfrak{R}\beta\}$, and note that if $\beta\le_\mathfrak{R}\alpha$ and $\alpha>S(\beta)$ then $CS(\beta)<CS(\alpha)$ since $\alpha$ is bigger than every ordinal $\le_\mathfrak{R}\beta$. We now ``thin" $D$: let $$E=\{\delta\in D: \forall \epsilon<\delta, S(\epsilon)<\delta\}.$$ $E$ remains club --- in particular, $S(\eta)$ is always countable when $\eta$ is --- and so it is enough to show that $E$ is an $\le_\mathfrak{R}$-antichain.

Fix elements $\alpha<\beta$ of $E$. Then $\beta>S(\alpha)$, so $\beta\not\le_\mathfrak{R}\alpha$. Conversely, if $\alpha\le_\mathfrak{R}\beta$ then we would have $CS(\alpha)<CS(\beta)$ since $\beta>S(\alpha)$, but this contradicts the observation above about $D$. So $E$ is in fact an $\le_\mathfrak{R}$-antichain.
\end{proof}


\section{Counterexamples to Vaught's conjecture}

We now show that in fact the global picture provided above for ordinals can be significantly generalized, and in fact yields a characterization of counterexamples to Vaught's conjecture.

A key property we used implicitly throughout Section 2 was the fact that forcing adds no new ordinals. This property holds of counterexamples to Vaught's conjecture as well, in a sense studied by Knight, Montalb\'an, and the author \cite{KMS16} (and independently by Kaplan and Shelah \cite{KS16}):

\begin{defn} A {\em generically presentable structure} is a pair $(\nu,\mathbb{P})$ such that $\mathbb{P}$ is a forcing notion, $\nu$ is a $\mathbb{P}$-name, and the evaluation of $\nu$ in a generic extension $V[G]$ by $\mathbb{P}$ is independent of $G$, up to isomorphism, in the sense that $$\vdash_{\mathbb{P}^2}\nu[G]\cong\nu[H].$$
\end{defn}

All isomorphism-invariant data about possible evaluations of a generically presentable structure is already known by $V$ by definition. If $T$ is a counterexample to Vaught's conjecture, these are all the models there are in generic extensions:

\begin{prop}[Knight, Montalb\'an, S. \cite{KMS16}]\label{genpresVaught} Suppose $T$ is a counterexample to Vaught's conjecture. Then if $G$ is $\mathbb{P}$-generic and $M\in V[G]$ is a model of $T$, we have $M=\nu[G\cap \mathbb{P}_{\le p}]$ for some  generically presentable structure of the form $(\nu,\mathbb{P}_{\le p})$ for $p\in\mathbb{P}$.
\end{prop}

Briefly, this is because otherwise we can force to add a perfect set of non-isomorphic evaluations of $\nu$, and being a counterexample to Vaught's conjecture is absolute (see Morley \cite{Mor70}). When $T$ is a counterexample to Vaught's conjecture, we will write ``$GM(T)$" to denote the class of generically presentable models of $T$.

On the other side of things, we have theories with a perfect set of nonisomorphic models. For such theories the following construction will prove useful:

\begin{cons}\label{perfectcase} Suppose $P$ is a perfect set of reals coding pairwise nonisomorphic structures, and conflate $P$ with the real coding the underlying tree of $P$. We define an ASR $\mathfrak{W}_P$ by setting $\mathfrak{W}(e, \mathcal{M},\mathcal{N})$ iff one of the following holds:\begin{itemize}
\item $\mathcal{N}$ is coded by some real $f\in P$, and $\Phi_{e+1}^{f\oplus P}\cong \mathcal{M}$.
\item $e=0$ and $\mathcal{M}\cong\mathcal{N}$.
\end{itemize}

Note that the relation ``$b$ and $a$ are codes for structures $\mathcal{B},\mathcal{A}$ with $\mathfrak{W}_P(e, a, b)$" is $\Sigma^1_1$ in $e, a, b$ (with parameter $P$).
\end{cons}

We are now ready to give our first characterization of counterexamples to Vaught's conjecture:

\begin{thm}\label{firstchar} Suppose $T$ is a theory with uncountably many countable models. Then the following are equivalent:\begin{enumerate}

\item $T$ is a counterexample to Vaught's conjecture.

\item For each ASR $\mathfrak{R}$, there is some countable ordinal $\alpha$ such that for each model $\mathcal{M}$ of $T$, the transitive collapse of $\{\beta: \beta$ is an ordinal $\mathfrak{R}$-reducible to $\mathcal{M}\}$ (this is $CS_R(\mathcal{M})$) is $<\alpha$.

\item For each ASR $\mathfrak{R}$, there is no embedding of $\omega_1$ into the $\mathfrak{R}$-degrees of countable models of $T$.
\end{enumerate}
\end{thm}

\begin{proof} Construction \ref{perfectcase} provides $(3)\rightarrow (1)$  and $(2)\rightarrow(1)$. Supposing $T$ is not a counterexample to Vaught's conjecture, let $P$ be a perfect set of nonisomorphic models of $T$. Then the ASR $\mathfrak{W}_P$ given by Construction \ref{perfectcase} witnesses the failure of both $(3)$ and $(2)$:\begin{itemize}

\item  To show that $(3)$ fails, let $\{f_\eta:\eta<\omega_1\}\subseteq P$ such that $\{deg(f_\eta\oplus P):\eta<\omega_1\}$ is a strictly increasing sequence of Turing degrees; letting $\mathcal{M}_\eta$ be the structure coded by $f_\eta$ gives us an embedding of $\omega_1$ into the $\mathfrak{W}_P$-degrees of models of $T$.

\item To show that $(2)$ fails, note that if $\mathcal{A}\models T$ is coded by $r\in P$, then $\mathcal{A}$ is $\mathfrak{W}_P$-above every ordinal below the next admissible above $r\oplus P$, and these ordinals are cofinal in $\omega_1$.
\end{itemize}


\bigskip

For $(1)\rightarrow (2)$, we may compute in $V$ the ordinal $CS_\mathfrak{R}(\mathcal{M})$ for each $\mathcal{M} \in GM(T)$, and (since there are only $\omega$-many indices for $\mathfrak{R}$-reductions) we have $CS_\mathfrak{R}(\mathcal{M})<\omega_1$ for all such $\mathcal{M}$. Now consider a forcing extension in which $\omega_1^V$ is countable. This universe also satisfies $CS_\mathfrak{R}(\mathcal{M})<\omega_1^V$ for each model $\mathcal{M}$ of $T$, since by Proposition \ref{genpresVaught} all such models come from elements of $GM(T)$. But then this model satisfies ``$CS_\mathfrak{R}$ is bounded strictly below $\omega_1$;" by absoluteness, so does $V$.

\bigskip

Finally, suppose $(1)$ holds and $(3)$ fails as witnessed by $\mathfrak{R}$. By $(1)$, fix some appropriately-definable countable-to-one surjection $s$ from $Mod(T)$ onto $\omega_1$ --- for example, height in the Morley tree or (for interesting contrast) place in the Muchnik order modulo an appropriate real as per Theorem 1.V3 of  \cite{Mon13}. Now define an ASR $\mathfrak{W}$ by $\mathfrak{W}(e, \mathcal{A}, \mathcal{B})$ iff \begin{itemize}

\item $\mathcal{B}\models T$, $\mathcal{A}$ is a countable ordinal, and

\item $\mathfrak{R}(e, \mathcal{C}, \mathcal{B})$ for some $\mathcal{C}\models T$ with $s(\mathcal{C})=\mathcal{A}$. 
\end{itemize}

(Without loss of generality, the language of $T$ is disjoint from the fixed language for ordinals, so we dodge transitivity trivially.) Since $(1)$ holds, we know $(2)$ must hold, so let $\alpha$ be the countable ordinal so guaranteed in the case of the ASR $\mathfrak{W}$. Since $(3)$ fails as witnessed by $\mathfrak{R}$, let $f$ be an injection from $\omega_1$ to the $\mathfrak{R}$-degrees of models of $T$. Since $s$ is countable-to-one, we may assume that $x<y$ implies $s(f(x))<s(f(y))$. But this means that $CS_\mathfrak{W}(f(x))<CS_\mathfrak{W}(f(y))$ for all $x, y \in \omega_1$, and in particular for some $z \in \omega_1$ we must have $CS_\mathfrak{W}(f(z))>\alpha$. But this contradicts the assumption on $\alpha$.
\end{proof}

\begin{rmrk} We have treated strong reducibilities in a very abstract way. However, naturally-occurring strong reducibilities have additional structure: they tend to come from relations on {\em individual reals}, in the same way that Medvedev reducibility comes from Turing reducibility. This leads naturally to the notion of a {\em fine ASR}, which is an ASR $\mathfrak{R}$ satisfying $\mathfrak{R}(e, \mathcal{A},\mathcal{B})$ iff $\mathfrak{S}(e, a, b)$ for all reals $a, b$ coding $\mathcal{A},\mathcal{B}$ respectively, for some projective relation $\mathfrak{S}\subseteq \omega\times\mathbb{R}\times\mathbb{R}$. 

This means that there are two additional clauses we can add to Theorem \ref{firstchar}, by restricting clauses (2) and (3) to specifically fine ASRs. The theorem still holds: while restricting attention to fine ASRs would seem to make $(2)$ and $(3)$ weaker, the same proofs go through since the ASRs produced by Construction \ref{perfectcase} are fine.
\end{rmrk}


\section{Further questions}

While better calibrating the set-theoretic hypotheses for the results above is a natural direction for further research, its interest seems likely to be limited to set theory. We end by mentioning two directions for future research which we hope will prove more broadly fruitful.

\subsection{Behavior on clubs}

Conspicuously missing from Theorem \ref{firstchar} is anything resembling Theorem \ref{clubordasr}. In a sense, it appears both too strong and too weak to provide a direct characterization of counterexamples to Vaught's conjecture.

\bigskip

On the ``too strong" side, the natural approach to producing a characterization of counterexamples to Vaught's conjecture analogous to Theorem \ref{clubordasr} would be to pull back along a definable bijection between $Mod(T)$ and $\omega_1$ when $T$ is a counterexample to Vaught's conjecture. However, no such bijection is known to exist:

\begin{qus} Suppose $T$ is a counterexample to Vaught's conjecture, and assume appropriate large cardinals. Is there a ``reasonably definable" bijection between $Mod(T)$ and $\omega_1$?
\end{qus}

In lieu of a positive answer to this question, merely being a counterexample to Vaught's conjecture doesn't seem to provide enough power to use a countably closed (pseudo-)ultrafilter on $Mod(T)$ for our purposes.

\bigskip

Conversely, on the ``too weak" side there is a version of Corollary \ref{thetaclub} which holds in a broader context --- namely, the existence of $\theta_{club}$ --- but here it holds {\em too } broadly. Suppose reasonable large cardinals exist, $\mathbb{K}$ is a reasonably-definable set of structures, and $\mathcal{U}$ is a reasonably-definable countably closed filter on $\mathbb{K}$ satisfying the dichotomy principle for reasonably-definable subsets of $\mathbb{K}$. Then by the argument in the proof of Theorem \ref{clubordasr} up to the ``thinning" step, \ZFC{} + appropriate large cardinals proves that $CS(-)$ is constant on a $\mathcal{U}$-large subset of $\mathbb{K}$ with value some countable ordinal $\theta_{\mathcal{U}}^\mathbb{K}$.

While this means that Corollary \ref{thetaclub} seems irrelevant to counterexamples to Vaught's conjecture, it does raise an interesting problem on its own:

\begin{prob} Assuming large cardinals, calculate $\theta_\mathcal{U}^\mathbb{K}$ for natural $\mathcal{U},\mathbb{K}$.
\end{prob}

\subsection{Near misses}

Characterizations of counterexamples to Vaught's conjecture can sometimes be modified to provide notions of ``near-miss" theories: theories which can be proved to exist, but which satisfy weak ``anti-Vaught" properties which are still interesting. For example, Montalb\'an \cite{Mon13} showed that any theory satisfying ``hyperarithmetic-is-recursive" is a counterexample to Vaught's conjecture; a natural response is to look for theories satisfying bounded versions of ``hyperarithmetic-is-recursive," like ``arithmetic-is-recursive," and indeed a theory satisfying ``arithmetic-is-recursive" was recently constructed by Andrews, Harrison-Trainor, Miller, and the author \cite{AHMS}.

\bigskip

Our characterization, Theorem \ref{firstchar}, also admits a notion of ``near-miss." Examining its proof, we see that the strong reducibility witnessing a theory's failure to be a counterexample to Vaught's conjecture is $\Sigma^1_1$ --- or rather, the predicate ``$e$ is an $\mathfrak{W}_P$-index for a reduction of the structure coded by $b$ to the structure coded by $a$" is $\Sigma^1_1$ in $e, a, b$ (with a parameter for $P$). The obvious way to improve this to $\Delta^1_1$ would be to find a perfect set $P$ of codes for pairwise nonisomorphic models of $T$ such that every model of $T$ has a code in $P$ --- that is, a perfect transversal of $(Mod(T),\cong)$ --- but the existence of such a $P$ is a very strong condition on $T$.\footnote{This was explained to the author by Danielle Ulrich at https://mathoverflow.net/a/265615/8133.} If $P$ is a perfect transversal of $T$, it follows (Theorem 6.4.4 of \cite{Gao08}) that $(Mod(T), \cong)$ is smooth --- that is, there is a Borel function reducing $(Mod(T),\cong)$ to $(\mathbb{R}, =)$. So any theory with uncountably-many models whose isomorphism problem is nontrivial in the sense of Borel reducibility will not have a perfect transversal; for example, by Theorem 1.3. of \cite{Mar07} it follows that no {\em non-small } theory has a perfect transversal. (And ``perfect" here can be replaced with ``Borel.")

\bigskip

This failure to easily improve the complexity bound on the reducibility witnessing that $T$ is not a counterexample to Vaught's conjecture motivates the following notion:

\begin{defn} Let $\Gamma$ be a pointclass. A theory $T$ is {\em $\Gamma$-intermediate for strong reducibilities} if $T$ has uncountably many countable models up to isomorphism, but there is no strong reducibility $\le_R\in \Gamma$ such that $\omega_1$ embeds into the $\equiv_R$-degrees of models of $T$.
\end{defn}

The author suspects that $\Delta^0_\alpha$-intermediate theories for fixed $\alpha$ will be of limited interest and that $\Pi^1_1$-intermediate theories are unlikely to exist. However, the case of $\Delta^1_1$-intermediate theories seems particularly interesting:

\begin{qus} Are there theories (first-order or in $\mathcal{L}_{\omega_1\omega}$) which are $\Delta^1_1$-intermediate for strong reducibilities?
\end{qus}





\begin{thebibliography}{GHHM13}

\bibitem[AS76]{AbSa76}
F.G. Abramson and Gerald Sacks.
\newblock Uncountable Gandy ordinals.
\newblock {\em Journal of the London Mathematical Society}, s2-14(3): pp. 387-392, 1976.

\bibitem[AHMS]{AHMS}
Uri Andrews, Matthew Harrison-Trainor, Joe Miller, Noah Schweber, and Maria Soskova.
\newblock The property ``arithmetic-is-recursive" on a cone.
\newblock {\em Journal of Mathematical Logic}, 21(3), 2021.

\bibitem[AK00]{AK00}
Chris Ash and Julia Knight.
\newblock {\em Computable structures and the hyperarithmetical hierarchy},
  volume 144 of {\em Studies in Logic and the Foundations of Mathematics}.
\newblock North-Holland Publishing Co., Amsterdam, 2000.

\bibitem[BF01]{BF01}
Joan Bagaria and Sy-David Friedman.
\newblock Generic absoluteness.
\newblock {\em Annals of pure and applied logic}, vol. 108: pp. 3-13, 2001.

\bibitem[Bar67]{Bar67}
Jon Barwise.
\newblock {\em Infinitary logic and admissible sets}.
\newblock Ph.D. thesis, Stanford University, 1967.

\bibitem[Gao08]{Gao08}
Su Gao.
\newblock {\em Invariant descriptive set theory}.
\newblock Chapman and Hall/CRC Pure and Applied Mathematics, 2008.

\bibitem[HL$\infty$]{HaLi}
Joel David Hamkins and Zhenhao Li.
\newblock On effectiveness of operations on countable ordinals.
\newblock {\em Unpublished (available at http://zhenhao-li.org/)}.

\bibitem[HLR13]{HLR13}
Joel David Hamkins, David Linetsky, and Jonah Reitz.
\newblock Pointwise definable models of set theory.
\newblock {\em Journal of symbolic logic}, 78(1): pp. 139-156, 2013.

\bibitem[Ha78]{Ha78}
Leo Harrington.
\newblock Analytic determinacy and $0^\sharp$.
\newblock {\em Journal of symbolic logic}, 43(4): pp. 685-693, 1978.

\bibitem[Hau95]{Hau95}
Kai Hauser.
\newblock The consistency strength of projective absoluteness.
\newblock {\em Annals of pure and applied logic}, 74: pp. 245-295, 1995

\bibitem[Jec03]{Jec03}
Thomas Jech.
\newblock {\em Set theory}.
\newblock Springer Monographs in Mathematics. Springer-Verlag, Berlin, 2003.
\newblock The third millennium edition, revised and expanded.

\bibitem[Kal12]{Kal12}
Iskander Kalimullin.
\newblock Algorithmic reducibilities of algebraic structures,
\newblock {\em Journal of logic and computation}, 22(4): pp. 831-843, 2012.

\bibitem[KS16]{KS16}
Itay Kaplan and Saharon Shelah.
\newblock Forcing a countable structure to belong to the ground model.
\newblock {\em Mathematical logic quarterly}, 62(6): pp. 530-546, 2016.

\bibitem[KMS16]{KMS16}
Julia Knight, Antonio Montalb\'an, and Noah Schweber.
\newblock Computable structures in generic extensions.
\newblock {\em Journal of symbolic logic}, 81(3): pp. 814-832, 2016.

\bibitem[Kri64]{Kri64}
Saul Kripke.
\newblock Admissible ordinals and the analytic hierarchy.
\newblock {\em Journal of symbolic logic}, 29(3): p. 162, 1964.

\bibitem[Mar07]{Mar07}
David Marker.
\newblock Borel complexity of isomorphism.
\newblock {\em Notre Dame Journal of Formal Logic}, 47(1): pp.93-97, 2007.

\bibitem[Mon13]{Mon13}
Antonio Montalb\'an.
\newblock A computability theoretic equivalent to Vaught's conjecture.
\newblock {\em Advances in mathematics}, vol.235: pp. 56-73, 2013.

\bibitem[Mor70]{Mor70}
Michael Morley
\newblock The number of countable models.
\newblock {\em Journal of symbolic logic}, 35(1): pp. 14-18, 1970

\bibitem[Pla66]{Pla66}
Richard Platek.
\newblock {\em Foundations of recursion theory}.
\newblock Ph.D. thesis, Stanford University, 1966.

\bibitem[Sac76]{Sac76} 
Gerald Sacks.
\newblock Countable admissible ordinals and hyperdegrees.
\newblock {\em Advances in mathematics}, 20(2): pp. 213-262, 1976

\bibitem[Sac90]{Sac90} 
Gerald Sacks.
\newblock {\em Higher recursion theory}.
\newblock Perspectives in mathematical logic, vol. 2, Berlin: Springer-Verlag, 1990.

\bibitem[Sch16]{Sch16}
Noah Schweber.
\newblock {\em Interactions between computability theory and set theory}.
\newblock Ph.D. thesis, University of California --- Berkeley, 2016.



\end{thebibliography}
\end{document}